\documentclass[12pt]{amsart}
\usepackage{graphicx,amssymb}
\usepackage{pb-diagram}
\usepackage[all]{xy}

\newtheorem{thm}{Theorem}[section]
\newtheorem{cor}[thm]{Corollary}
\newtheorem{lem}[thm]{Lemma}
\newtheorem{prop}[thm]{Proposition}

\theoremstyle{definition}
\newtheorem{defn}[thm]{Definition}

\theoremstyle{remark}
\newtheorem{rem}[thm]{Remark}
\newtheorem{exa}[thm]{Example}

\numberwithin{equation}{section}

\DeclareMathOperator{\GL}{{GL}}

\DeclareMathOperator{\Deg}{{deg}}
\DeclareMathOperator{\Div}{{div}}
\DeclareMathOperator{\DDiv}{{Div}}

\DeclareMathOperator{\DDeg}{{deg}}
\DeclareMathOperator{\Eq}{{eq}}

\DeclareMathOperator{\ind}{Ind}

\begin{document}

\title{Group Actions on Riemann-Roch Space}
\author{Angel Carocca}
\author{Daniela V\'asquez Latorre}
\address{Departamento de Matem\'aticas, Universidad de La Frontera, Casilla 54-D, Temuco, Chile}
\email{angel.carocca@ufrontera.cl}
\address{Departamento de Matem\'aticas, Universidad del Valle, 
Ciudad Universitaria Mel\'endez, Calle 13 \# 100-00,  C\'odigo Postal 760032, Cali, Colombia}

\email{daniela.vasquez@correounivalle.edu.co}
\subjclass[2010]{14H55, 14H51}
\keywords{Riemann surfaces, Divisors, Riemann-Roch Space}
\thanks{The  first author was  partially supported by Anillo ACT 1415 PIA-CONICYT.}

\begin{abstract}
Let $ \; G \; $ be a group acting on a compact Riemann surface $ \; {\mathcal X} \; $ and  $ \; D \; $ be a $ \; G$-invariant divisor on $\; {\mathcal X}. \; $ The action of $ \; G \; $ on $ \; {\mathcal X} \; $ induces a linear representation $ \; L_G(D) \; $ of $ \; G \; $ on the Riemann-Roch space associated to $ \; D.$
\vspace{1mm}\\
In this paper we give some results on the decomposition of $ \; L_G(D) \; $ as sum of complex irreducible representations of $ \; G, \; $ for
$ \; D \; $ an effective non-special $\; G$-invariant divisor. In particular, we give explicit formulae for the multiplicity of each complex irreducible factor in $ \; L_G(D) \; $. We work out some examples on well known families of curves.
\end{abstract}
\maketitle
\section{Introduction}

Let $ \; {\mathcal X} \; $  be a compact Riemann surface, and let $ \; G \; $  be a  group of automorphisms of $ \; {\mathcal X}. \; $

If $ \; D \; $ is a divisor on $ \; {\mathcal  X} \; $ which is stable under the action of $ \; G, \; $ then $ \; G \; $  acts on the Riemann-Roch space $ \; {\mathcal L}(D) \; $ associated to $ \; D. \; $

The problem of determining the decomposition of the induced linear representation $ \; L_G(D) \; $ of $ \; G \; $ on $ \; {\mathcal L}(D) \; $ as sum  of irreducible  representations of $ \; G \; $  was originally considered by A. Hurwitz \cite{h} in the case $ \; D \; $ a canonical divisor and $ \; G \; $ a cyclic group. C. Chevalley and A. Weil \cite{cw}, extended this result to any finite group and $ \; D \; $ a canonical divisor.
\vspace{1mm}\\
Since then, many authors have worked on this problem  for certain types of divisors. See for instance Borne \cite{Bo}, Ellingsrud and L$\phi$nsted \cite{E-L}, Kani \cite{Ka}, K\"och, \cite{Ko}, and Nakajima \cite{Na}. In the case where $ \; D \; $ is a non-special divisor, an equivariant Riemann-Roch formula was given for the character of $ \; L_G(D), \; $ 
see for instance \cite{Bo}. Also in this case, Joyner and Ksir \cite{J-K} gave explicit formula for the multiplicity of each rational irreducible factor when $ \; L_G(D) \; $ is a rational representation of $ \; G$.
\vspace{2mm}\\
In this paper we extend the Joyner-Ksir's results to the general case: that is, when $ \; L_G(D) \; $ is a complex representation of $ \; G $. We give explicit formulae for the multiplicity of each complex irreducible factor in the decomposition of $ \; L_G(D) \; $ as a sum of complex  irreducible representations of $ \; G. \; $ 
To illustrate this decomposition we give some examples of group actions on Riemann-Roch spaces for divisors on  well known families of curves.

\section{Group actions on Riemann surfaces}
First we  recall some facts   about actions of a finite group $ \; G \; $ on a compact Riemann surface $ \; {\mathcal X} \; $ of genus $ \; \mathtt{g}. \; $
The quotient space $\; \mathcal{X}_{G}:= \mathcal{X}/G\;$ is a Riemann surface and the quotient projection
$\Pi: \;\mathcal{X} \rightarrow \mathcal{X}_{G}\;$ is a branched cover. This cover may be partially characterized
by a vector of numbers $\;(\gamma ; m_{1}, \cdots , m_{r})\;$ where $\;\gamma \; $ is the genus of $\; \mathcal{X}_{G}, $ the integer  $ \; 0 \leq r \leq 2 \mathtt{g}  + 2 \;$ is the number of branch points $\{ Q_1 , \ldots , Q_r\} \subset \mathcal{X}_{G}$ of the cover, and for each $ \; 1 \leq j \leq r $ the integer  $ \; m_{j} > 1 \; $ is the order 
of the cyclic subgroup $ \; G_j  = \langle c_j \rangle \; $ of $ \; G \; $ stabilizing a  point  $ \; P_j \in \mathcal{X}$ with $\Pi(P_j)=Q_j$. We call $(\gamma ; m_{1}, \cdots , m_{r})$ the \textit{branching data} of $G$ on $ \; \mathcal{X}$. These numbers satisfy the Riemann-Hurwitz equation
\begin{equation}\label{rh}
\frac{2(\mathtt{g} -1)}{\vert G\vert }=2(\gamma-1)+\sum_{j=1}^{r}\left(1-\frac{1}{m_{j}}\right) .
\end{equation}
\noindent A $(2\gamma+r)-$tuple $\left(a_{1},\cdots, a_{\gamma},b_{1},\cdots,b_{\gamma}, c_{1},\cdots,c_{r}\right)$ of elements of $ \; G \; $ is called a \textit{generating vector of type } $(\gamma;m_{1},\cdots,m_{r})$  if
$$\label{gvector}
 G = \left\langle a_{1},\cdots, a_{\gamma},b_{1},\cdots,b_{\gamma}, c_{1},\cdots,c_{r} \; \; / \;  \prod_{i=1}^{\gamma}[a_{i},b_{i}]\prod_{j=1}^{r}c_{j} = 1 \; ,  \; \vert c_{j}\vert  = m_{j} \;  \mbox {for} \; j =1,...,r \; , \; {\mathcal R}  \right\rangle\\
$$
where $[a_i,b_i]=a_ib_ia_i^{-1}b_i^{-1} \; $ and $ \; {\mathcal R} \; $ is a set of appropriate relations on  $ \; \{ a_{1},\cdots, a_{\gamma},b_{1},\cdots,b_{\gamma}, c_{1},\cdots,c_{r}\}.$
\vspace{2mm}\\
Riemann's Existence Theorem then tells us that (see \cite{Brou})

\begin{thm}
The group $G$ acts on a compact Riemann surface $\; \mathcal{X} \; $ of genus $\; \mathtt{g}$ with branching data $(\gamma;m_{1},\cdots,m_{r})$ if and only if $G$ has a generating vector of type $(\gamma;m_{1},\cdots,m_{r})$ satisfying the Riemann-Hurwitz formula \eqref{rh}.\\
\end{thm}
\noindent
In order to formulate the results, we use the following notation:
Let $ \; G \; $ be a  group acting on a compact Riemann surface $ \; {\mathcal X} \; $ with branching data $\; (\gamma;m_{1},\cdots,m_{r})\; $
and generating vector $ \; \left(a_{1},\cdots, a_{\gamma},b_{1},\cdots,b_{\gamma}, c_{1},\cdots,c_{r}\right). \; $
For each $ \; P  \in  {\mathcal X} \; $  let $ \; G_P  = \langle c_P \rangle \; $ be the stabilizer of $ \; P \; $ in $ \; G, \; $ of order $ \; m_P \geq 1. \; $\\
The subgroup  $ \; G_P \; $  acts on the cotangent space $ \; {\mathcal X}(P) \; $ at $ \; P \; $ by a  $ \; {\mathbb C}$-character $ \; \omega_{P}. \; $ This is the ramification character of $ \; {\mathcal X} \; $ at $ \; P. \; $ 
Since  $ \; G_P \; $ is cyclic we have that $ \; \omega_{P} \; $ is a primitive $ \; m_P^{th}$-root of the unity. Particularly,  for $ \; P_j \; $ a branch point we will write $ \; G_{P_j} = G_j \; $ with order $ \; m_{P_j} = m_j \; $ and $ \; \omega_P = \omega_j \; $ the corresponding ramification character. 
\vspace{2mm}\\
For $ \; V \; $ a  complex irreducible representation of $ \; G, \; $ let $ \; N^V_{P\: k} \; $ be the number of times that  $ \; \omega_{P}^{k} \; $ is  an eigenvalue of $ \; V(c_P).$

\subsection{Rational and Analytical Representations}
The action of $ \; G \; $ on $\; \mathcal{X} \; $ induces two linear representations of $ \; G, \; $ the \textit{rational representation}
$ \; \rho_r : G \rightarrow \GL(H_1(\mathcal{X}, \mathbb{Z})\otimes {\mathbb Q}) \; $ and the \textit{analytic representation} $ \; \rho_a : G \rightarrow \GL(H^{1,0}(\mathcal{X}, \mathbb{C})). \; $  Both are related by
$$\rho_r \otimes {\mathbb C} \cong \rho_a \oplus \rho_a^{\ast} $$
where $ \; \rho_a^{\ast} \; $ is the complex conjugate of $ \; {\rho_a}.$
\vspace{2mm}\\
The multiplicity of each complex irreducible factor in
the decomposition of $ \; \rho_a \; $ as sum of complex irreducible  representations of $ \; G \; $ was given by Chevalley-Weil in \cite{cw}, as follows.

\begin{thm}\label{CH-W}
Let $ \; V \; $ be a non-trivial complex irreducible  representation of $ \; G. \; $ 
Then the multiplicity $ \; a(V) \; $ of  $ \; V \; $  in
the decomposition of $ \; \rho_a \; $ as sum of complex  irreducible  representations of $ \; G \; $ is given by
$$a(V) = \dim_{\mathbb C}(V)(\gamma - 1) + \displaystyle{\sum_{j=1}^{r}\sum_{k = 1}^{m_j -1}N_{j\:k}^{V}\left<-\frac{k}{m_j} \right>}  $$
where $ \; \displaystyle{\left< q \right> = q -[q]} \; $ is the fractional part of the rational number $ \; q. \;  \; $ 

Furthermore, for the trivial representation $ \; V_0 \; $ we have $ \; a(V_0) = \gamma. $
\end{thm}
For the rational representation $ \; \rho_r \; $ similar formulae was given by Broughton \cite{br} as follows.
\begin{thm} 
Let $ \; V \; $ be a non-trivial complex irreducible  representation of $ \; G. \; $ 
Then the multiplicity $ \; r(V) \; $ of  $ \; V \; $  in
the decomposition of $ \; \rho_r \otimes {\mathbb C} \; $ as sum of complex irreducible  representations is given by
$$r(V) = 2 (\gamma - 1)\dim_{\mathbb C}(V) +   \sum_{j=1}^{r}(\dim_{\mathbb C}(V) - \dim_{\mathbb C}(V^{G_j}))$$
where  $ \; \; V^{G_j} \; $ denotes the fixed subspace of $ \; V \; $ under the action of $ \; G_j. \; $\\
\end{thm}
With these results we obtain the following
\begin{cor}\label{corolario}
If $ \; V \; $ is a non-trivial absolutely  irreducible  representation of  $ \; G, \; $ then 
$$a(V) = a^{\ast}(V) = (\gamma - 1)\dim_{\mathbb Q}(V) + \displaystyle{\frac{1}{2}\left( \sum_{j=1}^{r}\dim_{\mathbb Q}(V) - \dim_{\mathbb Q}(V^{G_j})\right)}  $$
where $ \; a^{\ast}(V) \; $ 
is  the multiplicity of  $ \; V \; $  in
the decomposition of $ \; \rho_a^{\ast}.\; $ \\
\end{cor}
\subsection{The Ramification Module}
The following definition was introduced in \cite{Ka} (also see \cite{Bo}, \cite{J-K} and \cite{Na}).
 The \textit{ramification module} for the cover $\Pi: \;\mathcal{X} \rightarrow \mathcal{X}_{G}\;$  with branching data $\; (\gamma;m_{1},\cdots,m_{r})\; $
 and generating vector $ \; \left(a_{1},\cdots, a_{\gamma},b_{1},\cdots,b_{\gamma}, c_{1},\cdots,c_{r}\right) \; $ is defined by
$$
\Gamma_G =    \displaystyle{\sum_{j=1}^{r}\ind_{G_j}^G\left(\sum_{\alpha=1}^{m_j-1}\alpha\omega_{j}^{\alpha} \right)} .$$
The following result was proved by Kani \cite{Ka} and Nakajima  \cite{Na}.
\begin{thm} Let $ \; G \; $ be a group acting on $ \; {\mathcal X} \; $  and
$ \; \Gamma_G \; $ the associated ramification module. Then there is a unique $ \; G$-module $ \; \tilde{\Gamma}_G \; $ such that
$$ \Gamma_G = \tilde{\Gamma}_G^{\vert G\vert}.$$
\\
\end{thm}

Considering $ \; \tilde{\Gamma}^{\ast}_{G} \; $ the dual  $ \; G$-module of $ \; \tilde{\Gamma}_{G},  \; $ an interesting relationship between the analytical representation (character) and the representation (character) of $ \; G \; $ on $ \; \tilde{\Gamma}^{\ast}_G \; $ is given by the following result [\cite{Ka}, Theorem 2 and Corollary.]
\begin{thm}\label{analitica-gamma}
Let $ \; G \; $ be a  group acting on $ \; {\mathcal X}. \; $  If $ \; \chi_a \; $ is the character of the analytical representation and $ \; \chi_{\tilde\Gamma_G}^{\ast} \; $ is the character of the representation of $ \; G \; $ on
$ \; \tilde{\Gamma}^{\ast}_{G}, \; $ then
$$ \chi_a = \chi_0 + (\gamma -1)\chi_{reg} + \chi_{\tilde\Gamma_G}^{\ast} $$
where $ \; \chi_{reg} \; $ is the regular character and $ \; \chi_0 \; $ is the trivial character of $ \; G.$
\\
\end{thm}
As a  simple application of the above result we obtain a generalization of [\cite{J-K},  Proposition 5 and Corollary 6].
\begin{cor}\label{rema}  Let $ \; V \; $ be a complex irreducible  representation of $ \; G. \; $ Then
$$  \displaystyle{\langle  \chi_{\tilde{\Gamma}_G}, \chi_V \rangle} = \displaystyle{\left\{\begin{array}{ll}
\displaystyle{\langle \chi_{\tilde{\Gamma}_G}, \chi_0 \rangle}  = 0 & \mbox{if} \; \; V = V_0 \; \mbox{is the trivial representation:}\\
\\
a^{\ast}(V) + ( 1 - \gamma)\dim_{\mathbb C}(V)  & \mbox{if} \; \; V \; \mbox{is a non-trivial representation.}
\end{array}\right.}
$$
where $ \; \chi_V \; $ is the character of $ \; V \; $ and $ \; \langle\: , \:  \rangle \; $ is the usual inner product of characters.
\end{cor}
\begin{proof}
According to Theorem \ref{analitica-gamma}, observe that  if $ \; \chi_a^{\ast} \; $ is the dual character of $ \; \chi_a \; $ \; then
$$ \chi_a^{\ast} = \chi_0 + (\gamma -1)\chi_{reg} + \chi_{\tilde\Gamma_G}. $$
It follows that
\begin{enumerate}
\item the multiplicity of the trivial representation $ \; V_0 \; $ of $ \; G \; $ in $ \; \tilde{\Gamma}_G \; $ is
$$  \langle \chi_{\tilde{\Gamma}_G}, \chi_0 \rangle = \gamma + ( 1 - \gamma) - 1 = 0 \, ;$$
\item the multiplicity of any non-trivial complex irreducible representation $ \; V \; $ of $ \; G \; $ in $ \; \tilde{\Gamma}_G \; $ is 
$$  \displaystyle{\langle  \chi_{\tilde{\Gamma}_G}, \chi_V \rangle = a^{\ast}(V) + ( 1 - \gamma)\dim_{\mathbb C}(V) }.$$
\\
\end{enumerate}
\end{proof}
\section{Decomposition of $ \; L_G(D) \; $}
Let $ \; {\mathcal X} \; $ be a compact Riemann surface of genus $ \; \mathtt{g}\; $ and $ \; D \; $ a divisor on $ \; {\mathcal X}. \; $ We recall that the Riemann-Roch space associated to $ \; D \; $ is defined by
$$ {\mathcal L}(D) = \{ f \in {\mathbb C}^{\ast}({\mathcal X}) \; \; / \; \; \Div(f) \geq -D \} \cup \{ 0 \}$$
and the dimension of $ \; {\mathcal L}(D) \; $ is given by the  Riemann-Roch Theorem 
$$ \dim_{\mathbb C}({\mathcal L}(D)) = \Deg(D) - \mathtt{g}  + 1 + \dim_{\mathbb C}(\Omega(D))$$
where $ \; \Omega(D) = \{ \omega \; \; / \; \; \omega \; \mbox{is an abelian differential with } \; \Div(\omega) \geq D \} \cup \{ 0 \}.$
\vspace{2mm}\\
A divisor $ \; D \; $ is called \textit{non-special} if $ \; \dim_{\mathbb C}(\Omega(D)) = 0  $, or, equivalently, if $ \; \dim_{\mathbb C}({\mathcal L}(K - D)) = 0 \; $ for some canonical divisor $ \; K \; $ on $ \; {\mathcal X}$.
\\
\begin{rem}
As was mentioned earlier, if $ \; D \; $ is a divisor on $ \; {\mathcal  X} \; $ which is stable under the action of $ \; G, \; $ then $ \; G \; $  acts on the Riemann-Roch space $ \; {\mathcal L}(D) \; $ associated to $ \; D \; $ by the linear representation $ \; L_G(D).$
\vspace{2mm}\\
For each $P$ in ${\mathcal  X}$, consider the (basic) $ \; G$-invariant 
divisor given by
$$D_b(P) = \displaystyle{\frac{1}{m_P}\sum_{g \in G} g(P) \; \; \; \mbox{where} \; m_P = \vert G_P \vert .}
$$
Then the set of the basic divisors generates the group $ \; \DDiv({\mathcal X})^G \; $  of the  $ \; G$-invariant divisors on $ \; {\mathcal X}. \; $
\\
\end{rem}
We recall the definition of the equivariant degree,  as can be seen for example   in \cite{J-K}.
\begin{defn}
The equivariant degree is a map from $ \; \DDiv({\mathcal X})^G \; $  to the Grothendieck group $ \; R_{\mathbb K}(G) = \mathbb{Z}[G^{\ast}_{\mathbb K}], \; $
$$\DDeg_{\Eq} : \DDiv({\mathcal X})^G \to  R_{\mathbb K}(G)$$
defined by the following conditions:
\begin{enumerate}
\item $ \; \DDeg_{\Eq} \; $ is additive on the $ \; G$-invariant divisors of disjoint support;
\item If $ \; D = r_PD_b(P),  \; $ then
$$ \DDeg_{\Eq}(D) = \left\{ \begin{array}{cl}
\ind_{G_P}^G\left( \displaystyle{\sum_{k = 1}^{r_P}{\omega_P}^{-k}}\right)\; \; & \; \; \mbox{if} \; \; r_P > 0;
\\
\\
-\ind_{G_P}^G\left( \displaystyle{\sum_{k = 0}^{-(r_P+1)}{\omega_P}^{k}}\right)\; \; & \; \; \mbox{if} \; \; r_P < 0;
\\
\\
0\; \; & \; \; \mbox{if} \; \; r_P = 0.

\end{array}\right.$$
where $ \; \omega_P \; $ is the ramification character of $ \; {\mathcal X} \; $ at $ \; P.$
\\
\end{enumerate}
\end{defn}
\noindent
Now we compute the multiplicity of any complex irreducible representation of $ \; G \; $ in the decomposition of $ \; \DDeg_{\Eq}(D), \; $ for $ \; D \; $ a positive multiple  of a  basic divisor.

\begin{prop}\label{propo}
Consider $ \; D = r_pD_b(P) \; $   with $ \; r_P > 0\; $ and put   $ \; r_P = l_p + s_Pm_P,  \; \; $ where $ \;  0 \leq l_P < m_P. \; $ 
If $ \; V \; $ is a complex irreducible representation of $ \; G, \; $ 
then
 the multiplicity $ \; d(V) \; $ of  $ \; V \; $  in
the decomposition of $ \; \DDeg_{\Eq}(D) \; $    as sum of complex  irreducible  representations of $ \; G \; $   is given by
$$d(V) =  s_p\dim_{\mathbb C} (V)  + \epsilon_P\left( \dim_{\mathbb C}(V) - \sum_{k = 0}^{m_P -(l_P+1)} N_{P\: k}^{V} \right)
 $$
 where $ \epsilon_P = 0 \; $ if $ \; l_P = 0 \; \; $ and $ \epsilon_P = 1 \; \; $ if $ \; l_P \ne 0.$
\end{prop}
\begin{proof} Since $ \; \omega_P \; $ is a primitive $ \; m_P^{th}$-root of the unity, we have
$$ \; \rho_{reg} = \ind_{\{1\}}^G ( \chi_0 ) = \ind_{G_P}^G\left(\ind_{\{ 1 \}}^{G_P} \chi_0 \right)= \ind_{G_P}^G \left(\displaystyle{\sum_{k = 1}^{m_P}w_{P}^k}\right) = \ind_{G_P}^G \left(\displaystyle{\sum_{k = 1}^{m_P}w_{P}^{-k}}\right).$$
With this we obtain 
 $$\ind_{G_P}^G\left( \displaystyle{\sum_{k = 1}^{r_P}{\omega_P}^{-k}}\right)  = 
 \ind_{G_P}^{G}\left( \displaystyle{\sum_{k = 1}^{l_P + s_Pm_P}{\omega_P}^{-k}}\right) =\displaystyle{\left\{\begin{array}{ll}
 s_P\rho_{reg} +  \ind_{G_P}^G\left( \displaystyle{\sum_{k = 1}^{l_P}{\omega_P}^{-k}}\right) & \mbox{if} \; \; l_P \ne 0 \\
 \\
 s_P\rho_{reg} & \mbox{if} \; \; l_P = 0 \\
 \end{array}\right.}$$
Furthermore,   if $ \; l_P \ne 0 \; $ we have
 $$ \;  \ind_{G_P}^G\left( \displaystyle{\sum_{k = 1}^{l_P}{\omega_P}^{-k}}\right) +  \ind_{G_P}^G\left( \displaystyle{\sum_{k = 0}^{m_P - (l_P +1)}{\omega_P}^{k}}\right) = \rho_{reg}$$
 and in this way
$$\displaystyle{\left\langle 
 \ind_{G_P}^G\left( \displaystyle{\sum_{k = 1}^{l_P}{\omega_P}^{-k}}\right) , V \right\rangle = s_p\dim_{\mathbb C} (V)  -
 \sum_{k = 0}^{m_P -(l_P+1)} N_{P\: k}^{V}} 
$$
Hence
$$d(V) = \displaystyle{\left\langle \ind_{G_P}^G\left( \displaystyle{\sum_{k = 1}^{r_P}{\omega_P}^{-k}}\right) , V
 \right\rangle} = s_p\dim_{\mathbb C} (V)  + \epsilon_P\left( \dim_{\mathbb C}(V) - \sum_{k = 0}^{m_P -(l_P+1)} N_{P\: k}^{V} \right)
 $$
 where $ \epsilon_P = 0 \; $ if $ \; l_P = 0 \; \; $ and $ \epsilon_P = 1 \; \; $ if $ \; l_P \ne 0.$
      \end{proof}
  \noindent
  The following result can be seen in [\cite{J-K}, Lemma 4] (also see \cite{Bo}).
\begin{lem} \label{equi}Let $ \; D \; $ be a $ \; G$-invariant non-special divisor on $ \; {\mathcal X} \; $ and $ \; \chi_{\mathcal L} \; $ the character of the representation $ \; L_G(G) \; $ of $ \; G. \; $ Then
$$\chi_{\mathcal L} = (1 - \gamma)\chi_{reg} + \DDeg_{\Eq}(D) - \chi_{\tilde\Gamma_G} $$
\end{lem}
\noindent
Now we are able to prove our main result.

   \begin{thm}\label{teo}   Let $ \; G \; $ be a  group acting on $ \; {\mathcal X} \; $ and $ \; \displaystyle{ D = \sum_{P \in  {\mathcal X} }r_PD_b(P) }\; $ be an effective non-special divisor on $ \; {\mathcal X} . \; $

For each $ \; P \in  {\mathcal X} \; $ write $ \; r_P = l_p + s_Pm_P  \; $ with $ \;  0 \leq l_P < m_P. \; $  If $ \;  V \; $ is a non-trivial complex irreducible representation of $ \; G, \; $  then the multiplicity $ m(V) $ of $V$ in the decomposition of  $ \; L_G(D) \; $ as sum of irreducible complex representations of $G$ is given by
  $$\displaystyle{m(V) =   \sum_{P \in {\mathcal X} } s_P\dim_{\mathbb C}(V) + \sum_{P \in {\mathcal X} } \epsilon_P\left( \dim_{\mathbb C}(V) - \sum_{k = 0}^{m_P -(l_P+1)} N_{P\: k}^{V} \right)  - a^{\ast}(V) }$$
  where $ \epsilon_P = 0 \; $ if $ \; l_P = 0 \; \; $ and $ \epsilon_P = 1 \; \; $ if $ \; l_P \ne 0. \; $
  
  Furthermore, for the trivial representation $ \; V_0 \; $ we have 
  $$m(V_0) = 1 - \gamma +  \sum_{P \in {\mathcal X} } s_P.$$
  \end{thm}
  \begin{proof} According to  Lemma \ref{equi} we have
  $$\chi_{\mathcal L} = (1 - \gamma)\chi_{reg} + \DDeg_{\Eq}(D) - \chi_{\tilde\Gamma_G}. $$
  Hence
  $$m(V) = \langle \chi_{\mathcal L}, \chi_V \rangle = (1 - \gamma)\dim_{\mathbb C}(V) + \langle \DDeg_{\Eq}(D), \chi_V \rangle - \langle \chi_{\tilde\Gamma_G},  \chi_V \rangle.$$\\
  Now, applying Corollary \ref{rema} we have
  $$  \displaystyle{\langle  \chi_{\tilde{\Gamma}_G}, \chi_V \rangle} = \displaystyle{\left\{\begin{array}{ll}
\displaystyle{\langle \chi_{\tilde{\Gamma}_G}, \chi_0 \rangle}  = 0 & \mbox{if} \; \; V = V_0 \; \mbox{is the trivial representation;}\\
\\
a^{\ast}(V) + ( 1 - \gamma)\dim_{\mathbb C}(V)  & \mbox{if} \; \; V \; \mbox{is a non-trivial representation.}
\end{array}\right.}
$$
Since $ \;  \DDeg_{\Eq} \; $ is additive on the $ \; G$-invariant divisors of disjoint support, we can apply the Proposition \ref{propo} to each $ \; r_pD_b(P), \; \; P \in  {\mathcal X}. \; $ In this way, we have
 $$\displaystyle{\left\langle \DDeg_{\Eq}(D) , \chi_V
 \right\rangle} = \sum_{P \in {\mathcal X} } s_p\dim_{\mathbb C} (V)  + \sum_{P \in {\mathcal X} } \epsilon_P\left( \dim_{\mathbb C}(V) - \sum_{k = 0}^{m_P -(l_P+1)} N_{P\: k}^{V} \right)
 $$
 with $ \epsilon_P = 0 \; $ if $ \; l_P = 0 \; \; $ and $ \epsilon_P = 1 \; \; $ if $ \; l_P \ne 0. \; $ 
 
 To complete the proof, note that if $ \; l_P \ne 1, \; $ then
 $$\dim_{\mathbb C}(V_0) - \sum_{k = 0}^{m_P -(l_P+1)} N_{P\: k}^{V_0} = 0$$
 and the result follows.
  \end{proof}

\begin{rem}\label{rema2} Let $ \; D = \pi^{\ast}(D_0) \; $ be a  divisor on  $ \; {\mathcal X} \; $ which is a pullback of an effective divisor $ \; D_0 = \displaystyle{\sum_{Q \in   {\mathcal X}_G} \alpha_Q Q}\; $ on  $ \; {\mathcal X}_G.\; $
\\
Then $$D = \pi^{\ast}(D_0) = \displaystyle{\sum_{Q \in   {\mathcal X}_G} \alpha_Q \sum_{P \in  \pi^{-1}(Q)} m_PP = \sum_{Q \in   {\mathcal X}_G} \alpha_Q \sum_{g \in  G}g(P) = \sum_{Q \in   {\mathcal X}_G} \alpha_Q m_PD_b(P).
}$$ fixing $ \; P \in \pi^{-1}(Q).$
Hence, with the notation of the Theorem \ref{teo}, we have $ \; l_P = 0 \; \, $ and $ \; \; s_P = \alpha_Q, \; $ for all $ \; Q \in  {\mathcal X}_G. \; $  \\
\end{rem}
Our last result of this section is a generalization of [\cite{J-K}, Theorems 1 and 2].
\begin{cor} Let $ \; D = \pi^{\ast}(D_0) \; $ be a  non-special divisor on  $ \; {\mathcal X} \; $ which is a pullback of an effective divisor $ \; D_0 \; $ on  $ \; {\mathcal X}_G.\; $
 If $ \;  V \; $ is a non-trivial complex irreducible representation of $ \; G, \; $  then the multiplicity $ m(V) $ of $V$ in the decomposition of  $ \; L_G(D) \; $ as sum of irreducible complex representations of $G$ is given by
  $$\displaystyle{m(V) =  \Deg(D_0)\dim_{\mathbb C}(V)  - a^{\ast}(V) }.$$
 Furthermore, for the trivial representation $ \; V_0 \; $ we have  $ \; m(V_0) = \Deg(D_0) + 1 - \gamma .$
 \vspace{2mm}\\
 In particular, if $ \; V \; $ is a non-trivial absolutely  irreducible representation of $ \; G, \; $ then
 $$m(V) = \dim_{\mathbb Q}(V)(\Deg(D_0) + 1 - \gamma)  - \displaystyle{\frac{1}{2}\left( \sum_{j=1}^{r}\dim_{\mathbb Q}(V) - \dim_{\mathbb Q}(V^{G_j})\right)}. $$
\end{cor}
\begin{proof} Let  $ \; D_0 = \displaystyle{\sum_{Q \in   {\mathcal X}_G} \alpha_Q Q}.\; $ 
According to Remark \ref{rema2} we have 
$$D = \pi^{\ast}(D_0) = \displaystyle{\sum_{Q \in   {\mathcal X}_G} \alpha_Q \sum_{P \in  \pi^{-1}(Q)} m_PP = \sum_{Q \in   {\mathcal X}_G} \alpha_Q \sum_{g \in  G}g(P) = \sum_{Q \in   {\mathcal X}_G} \alpha_Q m_PD_b(P).
}$$ fixing $ \; P \in \pi^{-1}(Q) . $   
\\
Now applying Theorem \ref{teo} with  $ \; l_P = 0 \; \, $ and $ \; \; s_P = \alpha_Q, \; $   we have 
$$\displaystyle{m(V) =   \sum_{P \in {\mathcal X} } s_P\dim_{\mathbb C}(V) - a^{\ast}(V)  = \deg(D_0)\dim_{\mathbb C}(V) - a^{\ast}(V) }
$$
and $ \;  \; m(V_0) = \Deg(D_0) + 1 - \gamma .$
\vspace{2mm}\\
Finally,  according to Corollary \ref{corolario} for  $ \;  V \; $  a non-trivial absolutely  irreducible representation 
$$ a^{\ast}(V) = (\gamma - 1)\dim_{\mathbb Q}(V) + \displaystyle{\frac{1}{2}\left( \sum_{j=1}^{r}\dim_{\mathbb Q}(V) - \dim_{\mathbb Q}(V^{G_j})\right)}. $$
Then in this case $$m(V) = \dim_{\mathbb Q}(V)(\Deg(D_0) + 1 - \gamma)  - \displaystyle{\frac{1}{2}\left( \sum_{j=1}^{r}\dim_{\mathbb Q}(V) - \dim_{\mathbb Q}(V^{G_j})\right)}. $$

\end{proof}
\section{Examples}
In this section, 
to apply our results we give some examples of group actions on Riemann-Roch spaces for divisors on  well known families of curves. 

We first recall a well known fact:
\begin{rem} Let $ \; D \; $ be a divisor on $ \;  {\mathcal X}.  \; $ If $ \; \deg(D) > 2(\mathtt{g} -1), \; $ then $ \; D \; $ is non-special.\\
\end{rem}

\begin{exa} See [\cite{J-K}, Example 4]. Let $ \;  {\mathcal X} \; $ be the Klein quartic of genus $ \; \mathtt{g} = 3 \; $
$$\{ [X : Y :Z] \in {\mathbb C}{\mathbb P}^2 \; / \; X^3Y + Y^3Z + Z^3X = 0 \} .$$
Consider the automorphisms of  ${\mathcal X}$ given by $$ \tau[X : Y :Z] = [\eta X : \eta^4 Y : \eta^2 Z] \; \; \mbox{and} \; \; \sigma[X : Y : Z]= [Y : Z : X]$$
where $ \; \eta \; $ is a primitive seventh root of the unity.  The group $ \; G = \langle \tau , \sigma \rangle \cong \langle \tau \rangle \rtimes \langle \sigma \rangle \; $ has order $ \; 21 \; $ and has characters table given by
$$\begin{array}{|c|c|c|c|}
\hline
& \mbox{degree} & \tau & \sigma\\
\hline
\chi_0 & 1 & 1 & 1\\
\hline
\chi_1 & 1 & 1 & \zeta\\
\hline
\chi_2 & 1 & 1 & \zeta^2\\
\hline
\chi_3 & 3 & \eta + \eta^2 + \eta^4 & 0 \\
\hline
\chi_4 & 3 & \eta^3 + \eta^5 + \eta^6 & 0 \\
\hline
\end{array}
$$
with$ \; \zeta \; $ a primitive  cube root of the unity. The  complex irreducible  representations associated to $ \; \chi_3 \; \; $ and $ \; \; \chi_4 \; $ are respectively 
$$V_3(\tau) = \left( \begin{array}{lll}
\eta & 0 & 0 \\
0 & \eta^2 & 0\\
0 & 0 & \eta^4
\end{array} \right) \; \; \; ; \; \; \; V_3(\sigma) = \left( \begin{array}{lll}
0 & 0 & 1 \\
1 & 0 & 0\\
0 & 1 & 0
\end{array} \right) \equiv \left( \begin{array}{lll}
1 & 0 & 0 \\
0 & \zeta & 0\\
0 & 0 & \zeta^2
\end{array} \right)
$$ 
$$V_4(\tau) = \left( \begin{array}{lll}
\eta^3 & 0 & 0 \\
0 & \eta^5 & 0\\
0 & 0 & \eta^6
\end{array} \right) \; \; \; ; \; \; \; V_4(\sigma) = \left( \begin{array}{lll}
0 & 0 & 1 \\
1 & 0 & 0\\
0 & 1 & 0
\end{array} \right) \equiv \left( \begin{array}{lll}
1 & 0 & 0 \\
0 & \zeta & 0\\
0 & 0 & \zeta^2
\end{array} \right).
$$ 
The point $ \; P = [1 : \zeta : \zeta^3] \; $ is fixed by $ \; H_P = \langle \sigma \rangle. \; $ \\
Consider the non-special divisor $ \; D = D_b(P) = \displaystyle{\frac{1}{3}\sum_{g \in G}g(P)} \; $ of degree $ \; 7. \; $ \\
Then $ \; r_P = 1, \; l_P = 1 \; \; $ and $ \; s_P = 0. \; $ Also
$$\dim_{\mathbb C}({\mathcal L}(D)) = \Deg(D) - \mathtt{g} + 1 = 5.$$
Since $ \; m_P = 3 \; \; $ and $ \; l_P = 1, \; $ we have $ \; m_P - (l_P + 1 ) = 1 .\; $ In this way  for each $ \; V_j \; $ we obtain
$$\displaystyle{\sum_{k=0}^{1}N_{P \: k}^{V_1} = 1 \; \; ; \; \; \sum_{k=0}^{1}N_{P \: k}^{V_2} = 0 \; \; ; \; \; \sum_{k=0}^{1}N_{P \: k}^{V_3} = 2 \; \; ; \; \; \sum_{k=0}^{1}N_{P \: k}^{V_4} = 2}
$$
The analytic representation of $ \; G \; $ associated to the action on $ \;  {\mathcal X} \; $ is $ \; \rho_a = V_3. \; $ With this $ \; a^{\ast}(V_1) = 0, \; a^{\ast}(V_2) = 0, \; a^{\ast}(V_3) = 0 \; $ and $ \; a^{\ast}(V_4) = 1.$
\\
Applying  Theorem \ref{teo} we have
$$m(V_0) = 1, \; m(V_1) = 0, \; m(V_2) = 1, \; m(V_3) = 1 \; \; \mbox{and} \; m(V_4) = 0.$$
Finally we conclude
$$L_G(D) \cong V_0 \oplus V_2 \oplus V_3 \, .$$
\end{exa}

\medskip

\begin{exa} Let $ \; p \geq 5 \; $ be a prime number. \\
Consider  the Fermat curve   
$$ \;  {\mathcal X} \;   = \{ [X : Y :Z] \in {\mathbb C}{\mathbb P}^2 \; / \; X^p + Y^p + Z^p = 0 \} $$
of genus $ \; \mathtt{g} = \displaystyle{\frac{(p-1)(p-2)}{2}} \; $ and the automorphism of  $ \;  {\mathcal X} \; $ defined  by $ \; \sigma [X:Y:Z] = [\omega X:Y:Z] \; $ where $ \; \omega $ is a primitive $ \; p^{th}$-root of the unity. 

For
$ \; G = \langle \sigma \rangle \; $ the branching data is $ \; (0; \underbrace{p,p, \dots, p}_{p}) \; $ and a generating vector is $ \; (\underbrace{\sigma, \sigma, \dots, \sigma}_{p}). \; $ The non-trivial representations $ \; \{V_1, V_2, \dots , V_{p-1} \} \; $  of $ \; G \; $ are  defined by $ \; \sigma \rightarrow w^{i} \; \; $ with $ \; 1 \leq i \leq p-1. \; $ 
\vspace{2mm}\\
Let $ \; \eta \; $ be a primitive $ \; 2p^{th}$-root of the unity and $ \; P = [0 : \eta : 1] \in {\mathcal X} . \; $ Then $ \; P \; $ is a fixed point by $ \; G. \; \; $
Consider $ \; D = (p(p-3) + 1)D_b(P). \; $  Then
$$\dim_{\mathbb C}({\mathcal L}(D)) = \Deg(D) -  \mathtt{g}  = 1 = \displaystyle{\frac{(p-1)(p-2)}{2}}.$$
Is it not difficult to prove that the set
$$\beta = \displaystyle{\left\{  F_{a\: b}= \frac{X^aZ^b}{(Y - \eta  Z)^{a+b}} \; \; /  \; \; 0 \leq a \leq p-3, \; 0 \leq b \leq p-3, \; \; a+b \leq p-3 \right\}}$$
is a basis of $ \; {\mathcal L}(D). \; $ For $ \; F_{a\: b} \in \beta \; $ the action of $ \; G \; $ is given by $ \; \sigma (  F_{a\: b}) =  \displaystyle{\frac{\omega^aX^aZ^b}{(Y - \omega  Z)^{a+b}}} .$
\\
Hence
$$L_G(D) \cong (p-2)V_0 \oplus (p-3)V_1 \oplus \cdots \oplus 2V_{p-4} \oplus V_{p-3}.$$\\

Now applying  Theorem \ref{CH-W} the analytic representation of $ \; G \; $ is
$$\rho_a \cong (p-2)V_1 \oplus (p-3)V_2 \oplus \dots \oplus 2V_{p-3} \oplus V_{p-2}. $$
In this way
$$\rho_a^{\ast} \cong V_2 \oplus 2V_3 \oplus \cdots \oplus (p-3)V_{p-2} \oplus (p-2)V_{p-1} $$
We will apply  Theorem \ref{teo} with  $ \; s_P = p - 3 \; \; $ and $ \; l_P = 1. \; $ We have 
$$m(V_0) = 1 + s_P = p-2 $$
$$m(V_j) = -a^{\ast}(V_j) + s_P + \dim_{\mathbb C}(V_j) - \displaystyle{\sum_{k=0}^{p-2}N_{P \: k}^{V_j}}$$
Finally
$$\displaystyle{m(V_j) =  p-2 - a^{\ast}(V_j)  -\sum_{k=0}^{p-2}N_{P \: k}^{V_j}  = \left\{\begin{array}{cl} p - 2 - j & \mbox{for} \; 1 \leq j \leq p - 3 \\ 
\\
0 &   \mbox{for} \; p - 2  \leq j \leq p - 1 \end{array} \right.} $$
\end{exa}

\bibliographystyle{amsplain}

\end{document}